\documentclass[12pt]{amsart}
\pdfoutput=1
 \usepackage{amssymb,amsmath,amsfonts,epsfig,latexsym}

 \newtheorem{theorem}{Theorem}[section]
\newtheorem{definition}[theorem]{Definition}
\newtheorem{proposition}[theorem]{Proposition}
\newtheorem{lemma}[theorem]{Lemma}

\theoremstyle{definition}
\newtheorem{remark}[theorem]{Remark}
\newtheorem{example}[theorem]{Example}

\def\N{\ensuremath{\mathbb{N}}}
\def\Z{\ensuremath{\mathbb{Z}}}

\def\P{\ensuremath{\mathbb{P}}}

\def\F{\ensuremath{\mathbb{F}}}
\def\cF{\ensuremath{\mathcal{F}}}
\def\R{\ensuremath{\mathbb{R}}}

\def\OO{\ensuremath{\mathcal{O}}}

\DeclareMathOperator{\interior}{int}
\def\<{\ensuremath{\langle}}
\def\>{\ensuremath{\rangle}}

\DeclareMathOperator{\Hom}{Hom}
\DeclareMathOperator{\HHom}{\mathcal{H}\mathit{om}}

\DeclareMathOperator{\Spec}{Spec}

\begin{document}

\title[Frobenius splittings of toric varieties]{Frobenius splittings of toric varieties}

\author[Payne]{Sam Payne}

\thanks{Supported by the Clay Mathematics Institute}

\begin{abstract}
We discuss a characteristic free version of Frobenius splittings for toric varieties and give a polyhedral criterion for a toric variety to be diagonally split.  We apply this criterion to show that section rings of nef line bundles on diagonally split toric varieties are normally presented and Koszul, and that Schubert varieties are not diagonally split in general.
\end{abstract}

\maketitle

\section{Introduction}

Fix an integer $q$ greater than one.  Let $T = \Spec \Z[M]$ be the torus with character lattice $M$, and let $N$ be the dual lattice.  Let $\Sigma$ be a complete fan in $N_\R$, with $X = X(\Sigma)$ the associated toric variety over $\Z$.  Multiplication by $q$ preserves the fan and maps the lattice $N$ into itself, and therefore gives an endomorphism
\[
F: X \rightarrow X.
\]
Each $T$-orbit in $X$ is a torus that is preserved by $F$, which acts by taking a point $t$ to $t^q$.  For example, if $X$ is projective space, then $F$ is given in homogeneous coordinates by
\[
[x_0 : \cdots : x_n] \mapsto [x_0^q : \cdots : x_n^q].
\]
If $q$ is prime and $k$ is the field with $q$ elements, then the restriction of $F$ to the variety $X_{k}$ is the absolute Frobenius morphism.  Pulling back functions by $F$ gives a natural inclusion of $\OO_X$-algebras $F^* : \OO_X \hookrightarrow F_* \OO_X$.

\begin{definition}
A splitting of $X$ is an $\OO_X$-module map $\pi: F_* \OO_X \rightarrow \OO_X$ such that the composition $\pi \circ F^*$ is the identity on $\OO_X$.
\end{definition}

\noindent Standard results from the theory of Frobenius splittings generalize in a straightforward way to these splittings of toric varieties.  See Section~\ref{preliminaries} for details.  

If $Y$ is a subvariety of $X$ cut out by an ideal sheaf $I_Y$ and $\pi(F_* I_Y)$ is contained in $I_Y$ then we say that $\pi$ is \emph{compatible} with $Y$.  If $Y$ is a toric variety embedded equivariantly in $X$, the closure of a subtorus of an orbit in $X$, then a splitting compatible with $Y$ induces a splitting of $Y$.  We say that $X$ is \emph{diagonally split} if there is a splitting of $X \times X$ that is compatible with the diagonal, for some $q$.  Such splittings are of particular interest;  by classic arguments of Mehta, Ramanan and Ramanathan, if $X$ is diagonally split then every ample line bundle on $X$ is very ample and defines a projectively normal embedding.

Our main result is a polyhedral criterion for a toric variety to be diagonally split.  Let $v_\rho$ denote the primitive generator of a \emph{ray}, or one-dimensional cone, in $\Sigma$.  Let $M_\R = M \otimes_\Z \R$, and let the \emph{diagonal splitting polytope} $\F_X$ be defined by
\[
\F_X = \{ u \in M_\R \ | \ -1 \leq \<u, v_\rho \> \leq 1 \mbox{ for all } \rho \in \Sigma \}.
\]
We write $\frac{1}{q}M$ for the subgroup of $M_\R$ consisting of \emph{fractional lattice points} $u$ such that $qu$ is in $M$.

\begin{theorem} \label{diagonal}
The toric variety $X$ is diagonally split if and only if the interior of $\F_X$ contains representatives of every equivalence class in $\frac{1}{q}M / M$.
\end{theorem}

While the existence of a compatible splitting of the diagonal in $X \times X$ implies that section rings of ample line bundles on $X$ are generated in degree one, compatible splittings of large semidiagonals in products of multiple copies of $X$ give further information on these section rings.  For example, if the union of $\Delta \times X$ and $X \times \Delta$ is compatibly split in $X \times X \times X$, where $\Delta$ is the diagonal in $X \times X$, then it follows from standard arguments that the section ring of each ample line bundle on $X$ is \emph{normally presented}, that is, generated in degree one with relations generated in degree two.  For fixed $n$ greater than one, let $\Delta_i$ be the large semidiagonal
\[
\Delta_i = X^{i-1} \times \Delta \times X^{n - i - 1},
\]
for $1 \leq i < n$.  

\begin{theorem} \label{semidiagonals}
Let $X$ be a diagonally split toric variety.  Then $\Delta_1 \cup \cdots \cup \Delta_{n-1}$ is compatibly split in $X^n$.
\end{theorem}

\noindent In particular, if $X$ is diagonally split then the union of $\Delta \times X$ and $X \times \Delta$ is compatibly split in $X \times X \times X$, so the section ring of any ample line bundle on $X$ is normally presented.  Analogous results hold for any finite collection of nef line bundles on $X$, as we now discuss.

For line bundles $L_1, \ldots, L_r$ on $X$, let $R(L_1, \ldots, L_r)$ be the section ring
\[
R(L_1, \ldots, L_r) = \! \! \!  \bigoplus_{\ \ \ (\alpha_1, \ldots, \alpha_r) \in \N^r} H^0(X, L_1^{\alpha_1} \otimes \cdots \otimes L_r^{\alpha_r}).
\]
We consider $R(L_1, \ldots, L_r)$ as a graded ring, where the degree of $H^0(X, L_1^{\alpha_1} \otimes \cdots \otimes L_r^{\alpha_r})$ is $\alpha_1 + \cdots + \alpha_r$.  In particular, the degree zero part of $R(L_1, \ldots, L_r)$ is $\Z$.  Recall that a graded ring $R$ is Koszul if the ideal generated by elements of positive degree has a linear resolution as an $R$-module.  See \cite{PolishchukPositelski05} for background on Koszul rings and further details.

\begin{theorem} \label{diagonal to Koszul}
Let $X$ be a complete, diagonally split toric variety, and let $L_1, \ldots, L_r$ be nef line bundles on $X$.  Then the section ring $R(L_1, \ldots, L_r)$ is normally presented and Koszul.
\end{theorem}

\noindent  In particular, if $X$ is diagonally split then the section ring of any ample line bundle on $X$ is normally presented and Koszul.  Well-known open problems ask whether every ample line bundle on a smooth projective toric variety gives a projectively normal embedding \cite{Oda-preprint} and, if so, whether its section ring is normally presented \cite[Conjecture~13.19]{Sturmfels96}.  When the section ring is normally presented, it is natural to ask whether it is also Koszul.  Addressing these questions and their analogues for singular toric varieties in as many cases as possible is one of the main motivations behind this work.

\begin{remark}
The section ring $R(L_1, \ldots, L_r)$ associated to a finite collection of line bundles is canonically identified with the section ring of the line bundle $\OO(1)$ on the projectivized vector bundle $\P(L_1 \oplus \cdots \oplus L_r)$, which is also a toric variety.  If $L_1, \ldots, L_r$ are nef and correspond to polytopes $P_1, \ldots, P_r$, then the Cayley sum is the polytope associated to $\OO(1)$.  Cayley sums have also appeared prominently in recent work related to boundedness questions in toric mirror symmetry \cite{BatyrevNill07b, BatyrevNill07, HNP}.
\end{remark}

\begin{remark}
Frobenius morphisms and their lifts to characteristic zero have been used powerfully in several other contexts related to the geometry of toric varieties, including by Buch, Lauritzen, Mehta and Thomsen to prove Bott vanishing and degeneration of the Hodge to de Rham spectral sequence \cite{BTLM97}, by Totaro to give a splitting of the weight filtration on Borel-Moore homology \cite{Totaro}, by Smith to prove global $F$-regularity \cite{Smith00}, by Brylinski and Zhang to prove degeneration of a spectral sequence computing equivariant cohomology with rational coefficients \cite{BrylinskiZhang03}, and by Fujino to prove vanishing theorems for vector bundles and reflexive sheaves \cite{Fujino07}.  Frobenius splittings have also played a role in unsuccessful attempts to show that section rings of ample line bundles on smooth toric varieties are normally presented \cite{Bogvad95}.  We hope that this work will help revive the insight of B\o gvad and others into the potential usefulness of Frobenius splittings as a tool for understanding ample line bundles on toric varieties.
\end{remark}

We conclude the introduction with an example illustrating Theorem~\ref{diagonal} for Hirzebruch surfaces.  As mentioned earlier, the proofs that section rings of ample line bundles on Schubert varieties are normally presented and Koszul via Frobenius splittings involved compatible splittings of semidiagonals in $(G/B)^n$.  It has been an open question for over twenty years whether Schubert varieties themselves are diagonally split (see \cite[Remark~3.6]{Ramanathan87} and \cite[p.~81]{BrionKumar05}).  The following example gives a negative answer; the Hirzebruch surface $F_3$ is a Schubert variety in the $G_2$-flag variety, and $F_3$ is not diagonally split.

\begin{remark}
To see that $F_3$ occurs as a Schubert variety in the $G_2$-flag variety, first note that for any $G$, $G/B$ is a $\P^1$-bundle over $G/P$, where $P$ is a minimal parabolic subgroup.  If $w = s_1 s_2$ is an element of length two in the Weyl group of $G$, and $P$ is the minimal parabolic corresponding to $s_1$, then $X_w$ is a $\P^1$-bundle over its image, which is a rational curve in $G/P$.  In particular, $X_w$ is a Hirzebruch surface.  Then $X_{s_1}$ is a rational curve in $X_w$ with self-intersection $\< \alpha_2, \alpha_1^\vee\>$, where $\alpha_i$ is the simple root corresponding to $s_i$, and $\alpha_i^\vee = 2 \alpha_i / \<\alpha_i, \alpha_i \>$.  See \cite[Section~2]{Kempf76} for details.  For $G_2$, we can choose coordinates identifying the root lattice with the sublattice of $\Z^3$ consisting of those $(a_1, a_2, a_3)$ such that $a_1 + a_2 + a_3 = 0$, with simple roots $\alpha_1 = (1,-1,0)$ and $\alpha_2 = (-1,2,-1)$.  Then $X_{s_1}$ is a curve of self-intersection $-3$ in $X_w$, and hence $X_w$ is isomorphic to $F_3$.  See also \cite{Anderson07} for a detailed study of the $G_2$-flag variety and its Schubert varieties.
\end{remark}

\begin{example}
Let $a$ be a nonnegative integer, and let $\Sigma$ be the complete fan in $\R^2$ whose rays are spanned by $(1,0)$, $(0,1)$, $(0,-1)$, and $(-1, a)$.  Then $X(\Sigma)$ is isomorphic to the Hirzebruch surface $F_a$, the projectivization of the vector bundle $\OO_{\P^1} \oplus \OO_{\P^1}(a)$ \cite[pp. 7--8]{Fulton93}.  Let $q \geq 2$ be an integer.  By Theorem~\ref{diagonal}, $X$ is diagonally split if and only if the fractional lattice points in the interior of $\F_X$ represent every equivalence class in $\frac{1}{q}\Z^2 / \Z^2$.  The polytopes $\F_X$ for different values of $a$ are shown below.

\vspace{5 pt}

\begin{center}
\scalebox{.7}{\includegraphics{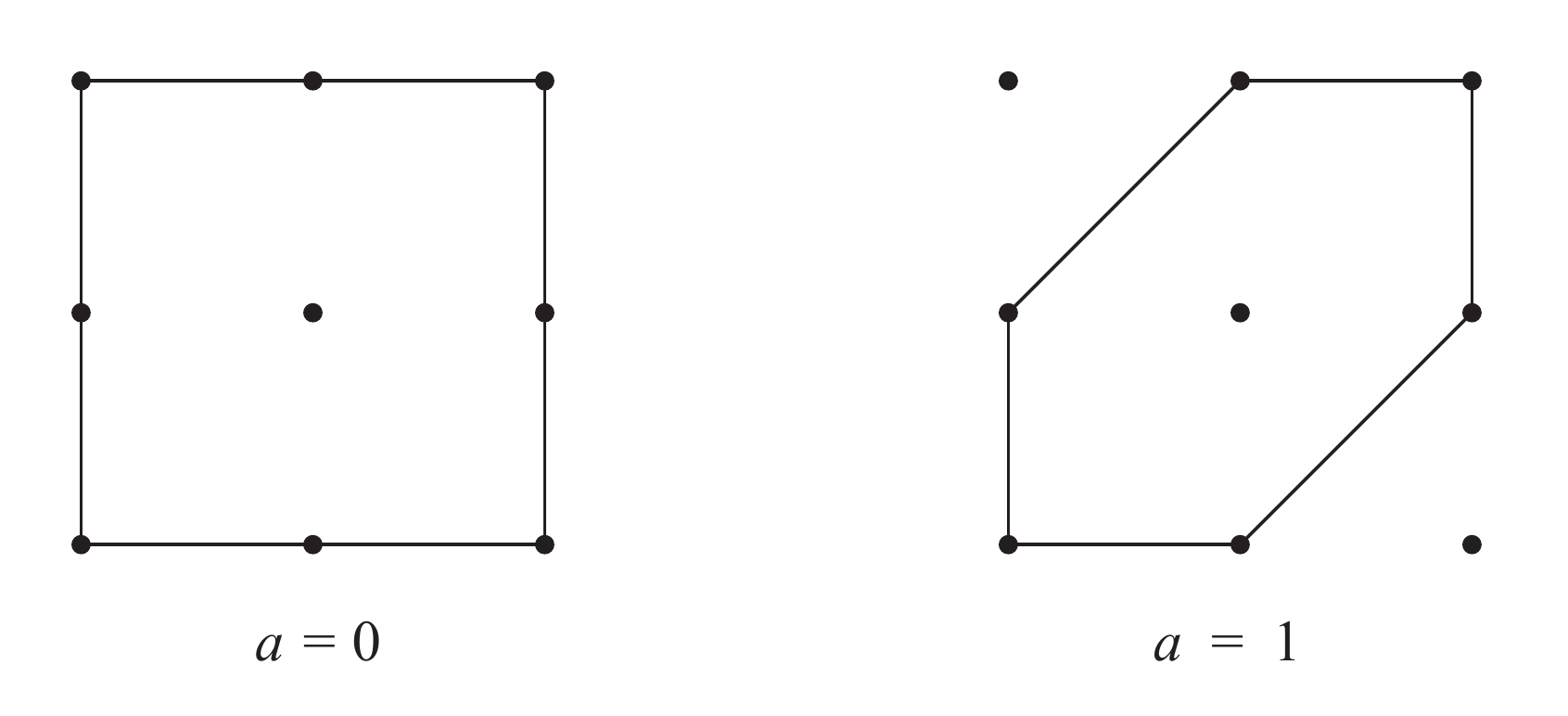}}
\end{center}

If $a$ is equal to zero or one, then the interior of $\F_X$ contains the half open unit square $[0,1) \times [0,1)$, which contains representatives of every equivalence class in $\frac{1}{q} \Z^2 / \Z^2$.  Therefore, $F_0$ and $F_1$ are diagonally split for all $q$.

\vspace{10 pt}

\begin{center}
\scalebox{.7}{\includegraphics{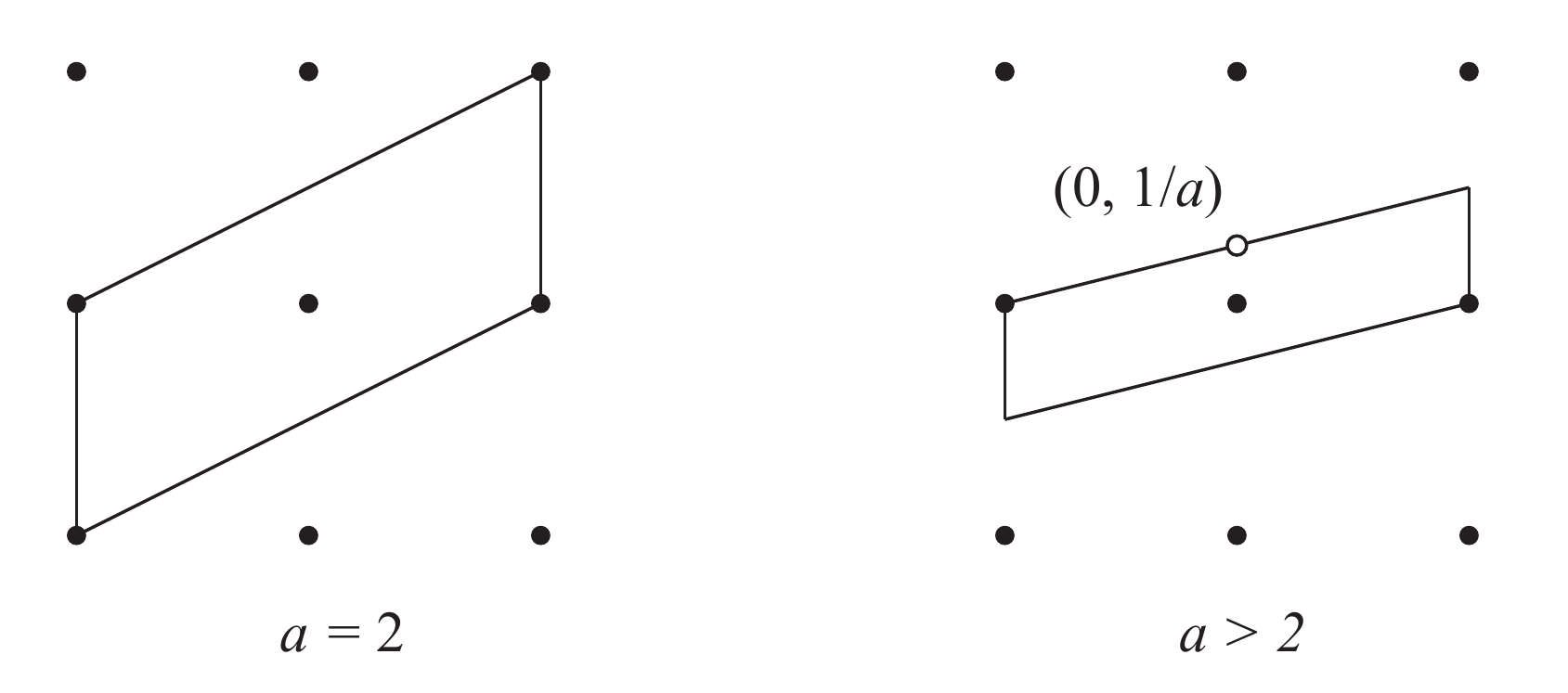}}
\end{center}

If $a$ is two, then $\F_X$ is the parallelogram with vertices $(\pm 1, 0)$, $\pm(1,1)$.  For $0 \leq m < q$, the interior of this intersection contains the fractional lattice points $(m/q, n/q)$ for $(m-q)/2 < n < (m + q)/2$.  If $q$ is odd, then these represent every equivalence class in $\frac{1}{q} \Z^2 / \Z^2$.  In particular, $F_2$ is diagonally split for $q$ odd.

\vspace{5 pt}

If $a$ is greater than two, then $\F_X$ is the parallelogram with vertices $(\pm 1, 0)$, $\pm(1, \frac{2}{a})$.  The only points in the interior of this intersection whose first coordinate is integral are of the form $(0, y)$ for $-1/a < y < 1/a$.  In particular, the equivalence class of $(0, \lfloor q/2 \rfloor / q)$ in $\frac{1}{q} \Z^2 / \Z^2$ is not represented by any point in the interior of this intersection.  Therefore, $F_a$ is not diagonally split for $a$ greater than two.
\end{example}

\noindent {\bf Acknowledgments.}  This work grew out of an expository talk at a mini-workshop on normality of lattice polytopes at Oberwolfach.  I thank the organizers, C.~Haase, T.~Hibi, and D.~Maclagan, as well as the other participants, for the chance to be part of such a stimulating event, and the MFO for their hospitality.  While preparing this paper, I benefited from helpful discussions with D.~Anderson, D.~Eisenbud, N.~Fakhruddin, B.~Howard, and A.~Paffenholz, and from the referee's thoughtful comments.  Finally, I am grateful to N.~Lauritzen and J.~Thomsen for pointing out that the methods in a first version of this paper can be used to give the example of the non-diagonally split Schubert variety presented here.

\section{Preliminaries} \label{preliminaries}

Frobenius splittings were introduced and developed by Mehta, Ramanathan, and their collaborators in the 1980s.  The original paper of Mehta and Ramanathan \cite{MehtaRamanathan85} is exceedingly well written and remains an excellent first introduction to the subject. Frobenius splittings were rapidly applied to give elegant unified proofs that all ample line bundles on generalized Schubert varieties of all types are very ample and give projectively normal embeddings whose images are cut out by quadrics \cite{RamananRamanathan85, Ramanathan87}.  Inamdar and Mehta, and independently Bezrukavnikov, later showed that the homogeneous coordinate rings of these embeddings are Koszul \cite{InamdarMehta94, Bezrukavnikov95}.  In characteristic zero, these results are deduced from the positive characteristic case using general semicontinuity theorems.  See the recent book of Brion and Kumar \cite{BrionKumar05} for a unified exposition of these results, along with further details, references, and applications.  On toric varieties, the Frobenius endomorphisms lift to endomorphisms over $\Z$, and it seems easiest and most natural to work independently of the characteristic using these lifted endomorphisms.  One feature of this approach is that we can prove results about section rings of toric varieties over $\Z$, or an arbitrary field, by producing a splitting of the diagonal in $X \times X$ for a single $q$.

\vspace{10 pt}

We begin by considering the structure of $F_* \OO_X$ as an $\OO_X$-module.  As a sheaf of groups, $F_* \OO_X$ evaluated on the invariant affine open set $U_\sigma$ associated to a cone $\sigma \in \Sigma$ is the coordinate ring $\Z[U_\sigma]$, which is usually identified with the semigroup ring $\Z[\sigma^\vee \cap M]$.  However, the module structure on $F_* \Z[U_\sigma]$ is different from the action of $\Z[U_\sigma]$ on itself.  For this reason, we identify $F_* \Z[U_\sigma]$ with the semigroup ring of fractional lattice points
\[
F_* \Z[U_\sigma] = \Z[\sigma^\vee \cap \textstyle{\frac{1}{q}}M],
\]
taking a monomial $x^u \in \Z[U_\sigma]$ to $x^{u/q}$.  The action of $\Z[U_\sigma]$ on $F_* \Z[U_\sigma]$ is then induced by the natural action of $M$ on $\frac{1}{q}M$, so
\[
x^u \cdot x^{u'} = x^{u + u'},
\]
for $u \in M$ and $u' \in \frac{1}{q}M$.    If $Y$ is a toric variety embedded equivariantly in $X$, then a splitting $\pi$ is compatible with $Y$ if and only if the induced map
\[
\Z[U_\sigma] \xrightarrow{\sim} \Z[\sigma^\vee \cap \textstyle{\frac{1}{q}}M] \xrightarrow{\pi} \Z[U_\sigma]
\]
maps $I_Y(U_\sigma)$ into $I_Y(U_\sigma)$ for every $\sigma \in \Sigma$.

We now summarize some basic properties of compatible splittings and their applications to section rings of ample line bundles.   Let $X$ be a complete toric variety, and let $L$ be a line bundle on $X$.  A splitting $\pi$ of $X$ makes $\OO_X$ a direct summand of $F_* \OO_X$ and hence $L$ a direct summand of $L \otimes F_* \OO_X$.  By the projection formula, $L \otimes F_* \OO_X$ is isomorphic to $F_*(F^* L)$, and we claim that $F^* L$ is isomorphic to $L^q$.  To see this, note that there is a $T$-invariant Cartier divisor $D$ such that $L$ is isomorphic to $\OO(D)$ \cite[Section~3.4]{Fulton93}.  The restriction of $D$ to an invariant affine open $U_\sigma$ is the divisor of a rational function $x^u$ for some $u \in M$, and hence the restriction of $F^*D$ to $U_\sigma$ is the divisor of $x^{qu}$.  It follows that $F^*L$ is isomorphic to $L^q$, as claimed.  Now, since cohomology commutes with direct sums, $\pi$ induces a split injection
\[
H^i(X,L) \hookrightarrow H^i(X, L^q),
\]
for every $i$.  Iterating this argument gives split injections of $H^i(X,L)$ in $H^i(X, L^{q^r})$ for all positive integers $r$.  In particular, if $H^i(X, L^{q^r})$ vanishes for some $r$, as is the case when $L$ is ample, then $H^i(X,L)$ vanishes as well.

The proofs of the following five propositions are essentially identical to the standard proofs of the analogous results for Frobenius splittings, and are omitted.  See \cite{BrionKumar05}; Proposition~1.2.1, Theorem~1.2.8, and Exercises~1.5.E.1, 1.5.E.2, and 1.5.E.3, respectively, for the case where the line bundles in question are ample.  The extensions to nef bundles can be deduced following the arguments in \cite{Inamdar94}, using the fact that any nef line bundle on $X$ is the pullback of an ample line bundle on some toric variety $X'$ under a proper birational toric morphism $f: X \rightarrow X'$.

\begin{proposition} \label{components}
Let $Y$ and $Y'$ be toric varieties equivariantly embedded in $X$.  If $Y \cup Y'$ is split compatibly in $X$ then $Y$, $Y'$, and $Y \cap Y'$ are split compatibly in $X$.
\end{proposition}

\begin{proposition}  \label{surjectivity and vanishing}
Let $Y$ be a compatibly split subvariety of $X$.  If $L$ is a nef line bundle on $X$ then the restriction map
\[
H^0(X,L) \rightarrow H^0(Y,L)
\]
is surjective and $H^1(X, I_Y \otimes L) = 0$.
\end{proposition}

\begin{proposition}
Let $L_1, \ldots, L_r$ be nef line bundles on $X$.  If the diagonal is compatibly split in $X \times X$  then the section ring $R(L_1, \ldots, L_r)$ is normally generated.
\end{proposition}

\begin{proposition}  \label{normal presentation}
Let $L_1, \ldots, L_r$ be nef line bundles on $X$.  If the union of $\Delta \times X$ and $X \times \Delta$ is compatibly split in $X \times X \times X$ then the section ring $R(L_1, \ldots, L_r)$ is normally presented.
\end{proposition}

\begin{proposition} \label{Koszul}
Let $L_1, \ldots, L_r$ be nef line bundles on $X$.  If $\Delta_1 \cup \cdots \cup \Delta_{n-1}$ is compatibly split in $X^n$ for every $n$ then the section ring $R(L_1, \ldots, L_r)$ is normally presented and Koszul.
\end{proposition}

\section{Canonical splittings}

Every toric variety has a splitting, and among all splittings of $X$ there is a unique one that extends to every toric compactification $X' \supset X$ and lifts to every proper birational toric modification $X'' \rightarrow X'$ of such a compactification.  If $q$ is prime and $k$ is the field with $q$ elements, then the restriction of this splitting to $X_k$ is the unique Frobenius splitting that is canonical in the sense of Mathieu \cite[Chapter~4]{BrionKumar05}.  We now describe this \emph{canonical splitting}, starting with its restriction to the dense torus $T$.

Let $\pi_0$ be the map of $\Z[T]$-modules from $F_* \Z[T]$ to $\Z[T]$ given by
\[
\pi_0 (x^u) = \left\{
\begin{array}{ll} x^u & \mbox{ if } u \in M, \\
			 0 & \mbox{ otherwise.}
			 \end{array}
			 \right.
\]			
The pullback map $F^*: \Z[T] \rightarrow F_* \Z[T]$ is induced by the inclusion of $M$ in $\frac{1}{q}M$; if $q$ is prime and $k$ is the field with $q$ elements, then the induced map $k[T] \rightarrow F_* k[T]$ may be identified with the inclusion of $k[T]$ in $k[T]^{1/q}$.

In particular, $\pi_0 \circ F_*$ is the identity, and hence gives a splitting of $T$.  

\begin{proposition}
For any toric variety $X$, $\pi_0$ extends to a splitting of $X$.
\end{proposition}

\begin{proof}
The composition $\pi_0 \circ F^*$ is the identity, and for each affine open $U_\sigma$, $\pi_0$ maps $\Z[\sigma^\vee \cap \frac{1}{q}M]$ into $\Z[\sigma^\vee \cap M]$.
\end{proof}

\noindent Properties of this canonical splitting $\pi_0$ are closely related to Smith's proof that toric varieties are globally $F$-regular \cite[Proposition~6.3]{Smith00}.  

\begin{proposition}
The canonical splitting $\pi_0$ is compatible with every $T$-invariant subvariety.
\end{proposition}

\begin{proof}
First we claim that $\pi_0$ is compatible with the union of the $T$-invariant divisors.  To see this, note that the ideal sheaf $I$ of the union of the invariant divisors is given by
\[
I(U_\sigma) = \Z[ \interior(\sigma^\vee) \cap M],
\]
where $\interior(\sigma^\vee)$ is the interior of $\sigma^\vee$.  If $u$ is a fractional lattice point in the interior of $\sigma^\vee$ then $\pi_0(x^u)$ is either zero or $x^u$, and so is contained in $I(U_\sigma)$, which proves the claim.  The proposition then follows from Proposition~\ref{components}, since every $T$-invariant subvariety is an intersection of invariant divisors.  
\end{proof}

\noindent However, if $X$ is positive dimensional then the canonical splitting $\pi_0$ of $X \times X$ is not compatible with the diagonal $\Delta$.  To see this, observe that if $u \in \frac{1}{q}M$ is not in $M$, then $1 - x^u \otimes x^{-u}$ is in $F_* I_\Delta$, but $\pi_0 (1 - x^{u} \otimes x^{-u}) = 1$, which is not in $I_\Delta$. To apply the standard techniques relating splittings to section rings of ample line bundles discussed in Section~\ref{preliminaries}, we must look for other splittings of $X \times X$, and $X^r$ for $r$ greater than two, that are compatible with the diagonal and the union of the large semidiagonals, respectively.  

\section{Splittings of diagonals} \label{diagonal section}

We now describe the space of all splittings of a toric variety and use this description to characterize diagonally split toric varieties.  First, it is helpful to consider the structure of $F_* \OO_X$ as an $\OO_X$-module in more detail.

Recall that an equivariant structure, or $T$-linearization, on a coherent sheaf $\cF$ on $X$ is an isomorphism of sheaves on $T \times X$,
\[
\varphi: \mu^*\cF \rightarrow p^*\cF,
\]
where $\mu : T \times X \rightarrow X$ is the torus action and $p$ is the second projection, that satisfies the usual cocycle condition \cite[Section~2.1]{BrionKumar05}.  For example, the natural equivariant structure on $\OO_X$ is given by $1 \otimes x^u \mapsto x^{-u} \otimes x^u$.  In general, the push forward of an equivariant sheaf under an equivariant morphism does not carry a natural equivariant structure.  However, the equivariant endomorphism $F$ has the property that $F_* \OO_X$ is \emph{equivariantizable} \cite{Bogvad98, Thomsen00}; it is possible to choose an equivariant structure as follows.  First, choose representatives $u_1, \ldots, u_s$ of the cosets in $\frac{1}{q}M /M$.  Let $\varphi$ be the map from $\mu^* F_* \OO_X$ to $p^* F_* \OO_X$ that takes $1 \otimes x^{u}$ to $x^{u_i - u} \otimes x^{u}$, for $u$ in the coset $u_i + M$.  It is straightforward to check that $\varphi$ is an isomorphism and gives an equivariant structure on $F_* \OO_X$, as required.

\vspace{10 pt}

A splitting of $X$ restricts to a splitting of $T$, and two splittings of $X$ agree if and only if they agree on $T$, so we describe the space of all splittings of $X$ in terms of splittings of $T$ that extend to $X$, as follows.  For fractional lattice points $a \in \frac{1}{q} M$, let
\[
\pi_a : F_*\Z[T] \rightarrow \Z[T]
\] 
be the map given by
\[
\pi_a (x^u) = \left\{
\begin{array}{ll} x^{a+u} & \mbox{ if } a + u \mbox{ is in } M, \\
			 0 & \mbox{ otherwise.}
			 \end{array}
			 \right.
\]

\begin{lemma}
The set of maps $\pi_a$, for $a$ in $\frac{1}{q}M$, is a $\Z$-basis for $\Hom(F_* \Z[T], \Z[T])$.
\end{lemma}

\begin{proof}
The maps $\pi_a$ are independent, and the free generators $x^{u_1}, \ldots, x^{u_s}$ for $F_* \Z[T]$ can be sent to an arbitrary $s$-tuple of elements of $\Z[T]$ by a suitable linear combination of the maps $\pi_a$.
\end{proof}

\noindent If we choose an equivariant structure for $\Hom(F_* \Z[T], \Z[T])$, as above, then the maps $\pi_a$ form a $T$-eigenbasis.  Therefore, a rational section
\[
\pi = c_1 \pi_{a_1} + \cdots + c_r \pi_{a_r}
\]
of the sheaf $\HHom(F_* \OO_X, \OO_X)$, with each $c_i$ nonzero, extends to $X$ if and only if each $\pi_{a_i}$ is regular on $X$.  For a ray, or one-dimensional cone $\rho$ in $\Sigma$, we write $v_\rho$ for the primitive generator of $\rho$.

\begin{proposition} \label{affine extension}
Let $U_\sigma$ be an affine toric variety.  Then $\pi_a$ is regular on $U_\sigma$ if and only if $\< a, v_\rho \>$ is greater than minus one for each ray $\rho$ in $\sigma$.
\end{proposition}

\begin{proof}
The map $\pi_a$ is regular on $U_\sigma$ if and only if it takes $\Z[\sigma^\vee \cap \frac{1}{q} M]$ into $\Z[\sigma^\vee \cap M]$.  Suppose $\<a, v_\rho\>$ is greater than minus one for each ray $\rho$ in $\sigma$ and $u$ is in $\sigma^\vee \cap \frac{1}{q}M$.  Either $\pi_a (x^u)$ is zero or $a + u$ is in $M$ and $\<u, v_\rho \>$ is a nonnegative integer for all rays $\rho$ in $\sigma$, and hence $a + u$ is in $\sigma^\vee$.  Therefore $\pi_a$ extends to $U_\sigma$.  Conversely, if $ \< a, v_\rho \> $ is less than or equal to minus one for some $\rho$, then it is straightforward to produce points $u \in \sigma^\vee$ such that $a + u$ is in $M$, but not in $\sigma^\vee$.  In this case, $\pi_a$ is not regular on $U_\sigma$.
\end{proof}

We follow the usual toric convention fixing $K = -\sum D_\rho$, the sum of the prime $T$-invariant divisors each with multiplicity minus one, as a convenient representative of the canonical class.  The polytope associated to a divisor $D = \sum d_\rho D_\rho$ is
\[
P_D = \{ u \in M_\R \ | \ \<u , v_\rho \> \geq -d_\rho \mbox{ for all } \rho\}.
\]
In particular, the polytope $P_{-K}$ associated to the anticanonical divisor is
\[
P_{-K} = \{ u \in M_\R \ | \ \<u, v_\rho \> \geq -1 \mbox{ for all } \rho \}.
\]
The interior of the polytope $P_{-K}$ controls the space of $\OO_X$-module maps from $F_* \OO_X$ to $\OO_X$ as follows.

\begin{proposition} \label{basis}
The set of maps $\pi_a$ for fractional lattice points $a$ in the interior of $P_{-K}$ is a basis for $\Hom(F_* \OO_X, \OO_X)$.
\end{proposition}

\begin{proof}
If $a$ is not in the interior of $P_{-K}$ then $\<a, v_\rho \>$ is less than or equal to minus one for some ray $\rho \in \Sigma$ and then $\pi_a$ is not regular on $U_\rho$.  Conversely, if $a$ is in the interior of $P_{-K}$, then $\pi_a$ extends to every invariant affine open subvariety of $X$, by Proposition~\ref{affine extension}, and therefore is regular on $X$.
\end{proof}

\begin{remark}
When $X$ is smooth, Proposition~\ref{basis} corresponds to the natural identification between $\Hom(F_* \OO_X, \OO_X)$ and $H^0(X, K_X^{1-q})$ given by duality for finite flat morphisms \cite[Section~1.3]{BrionKumar05}.
\end{remark}

\begin{proposition} \label{splitting condition}
A map $\sum c_a \pi_a$ in $\Hom(F_* \OO_X, \OO_X)$ is a splitting if and only if $c_0 = 1$.
\end{proposition}

\begin{proof}
Zero is the only lattice point in the interior of $P_{-K} \cap M$, so the image of $x^u$ under $\pi = \sum c_a \pi_a$ is equal to $c_0 x^u$ for $u \in M$.  In particular, since $F^*$ maps $x^u$ to $x^u$ in $F_* \OO_X$, $\pi \circ F^*$ is the identity if and only if $c_0$ is equal to one.
\end{proof}

\noindent By Proposition~\ref{splitting condition},  the set of splittings of $X$ is an affine hyperplane in $\Hom(F_* \OO_X, \OO_X)$.  For any subvariety $Y \subset X$, the condition that $\pi(F_* I_Y)$ is contained in $I_Y$ cuts out a linear subspace of $\Hom(F_* \OO_X, \OO_X)$.  So the set of splittings of $X$ that are compatible with $Y$ is an affine subspace of $\Hom(F_* \OO_X, \OO_X)$, which may be empty.  We now prove Theorem~\ref{diagonal}, which gives a necessary and sufficient condition for the space of splittings of $X \times X$ that are compatible with the diagonal to be nonempty.

\begin{proof}[Proof of Theorem~\ref{diagonal}]
Suppose $\pi = \sum c_{a,a'} \pi_{a,a'}$ is a splitting of $X \times X$ that is compatible with the diagonal.  Then the restriction of $\pi$ to the dense torus is a splitting compatible with the diagonal in $T \times T$.  For any $u \in \frac{1}{q} M$, we have
\[
1 - x^{u} \otimes x^{-u}
\]
in $F_* I_\Delta$, where $I_\Delta$ is the ideal of the diagonal in $T \times T$.  Since $\pi(1)$ is equal to one, the restriction of $\pi(x^{u} \otimes x^{-u})$ to the diagonal must also be equal to one.  Now the restriction of $\pi(x^{u} \otimes x^{-u})$ to the diagonal is a Laurent polynomial in $\Z[T]$ whose constant term is
\[
\sum_{ a \in [u]} c_{-a,a},
\]
where $[u]$ is the coset of $u$ in $\frac{1}{q} M / M$.  Since the polytope associated to $-K_{X \times X}$ is $P_{-K_X} \times P_{-K_X}$, there must be a representative $a$ of $[u]$ such that both $a$ and $-a$ are contained in the interior of $P_{-K_X}$, which means that $a$ is contained in the interior of the diagonal splitting polytope
\[
\F_X = P_{-K_X} \cap -P_{-K_X}.
\]

For the converse, suppose that every nonzero equivalence class $[u_i]$ in $\frac{1}{q}M/M$ has a representative $a_i$ in the interior of $\F_X$.  Then
\[
\pi_\Delta = \pi_0 + \sum_{i=1}^s \pi_{a_i,-a_i}
\]
is a splitting, and we claim that $\pi_\Delta$ is compatible with the diagonal.  To see this, note that the ideal of the diagonal in $U_\sigma \times U_\sigma$ is generated by the Laurent polynomials $1 - x^{u} \otimes x^{-u}$ for $u$ in $\sigma^\vee \cap M$.  Then $F_* I_\Delta$ is generated as a $\Z[U_\sigma \times U_\sigma]$-module by the
\[
x^b - x^b \cdot (x^{u} \otimes x^{-u}),
\]
as $b = (b_1, b_2)$ ranges over $\frac{1}{q}(M \times M)$ and $u$ ranges over $\frac{1}{q}M$.  Now the restriction of $\pi_\Delta (x^b)$ to the diagonal is $x^{b_1 + b_2}$ if $b_1 + b_2$ is in $M$ and zero otherwise.  In particular, the restriction of $\pi_\Delta(x^b - x^b \cdot (x^{u} \otimes x^{-u}))$ to the diagonal vanishes, as required.
\end{proof}

\begin{proof}[Proof of Theorem~\ref{semidiagonals}]
Suppose the diagonal is compatibly split in $X \times X$.  Then every nonzero equivalence class $[u_j]$ in $\frac{1}{q}M/M$ is represented by a fractional lattice point $a_j$ in the interior of $\F_X$, by Theorem~\ref{diagonal}.  

A splitting of $X^n$ is compatible with the union $\Delta_1 \cup \cdots \cup \Delta_{n-1}$ if it is compatible with each $\Delta_i$.  For $u \in M$, let $u^{(i)}$ denote the lattice point in $M^n$ whose only nonzero coordinate is the $i$-th one, which is equal to $u$.  The ideal of $\Delta_i$ is generated by the functions $1 - x^{u^{(i)}} \cdot x^{-u^{(i+1)}}$.  We claim that the splitting
\[
\pi = \pi_0 + \sum_{i=1}^{n-1} \sum_{j} \pi_{a_j^{(i)} - a_j^{(i+1)}}
\]
is compatible with $\Delta_i$ for $1 \leq i < n$, and hence with the union $\Delta_1 \cup \cdots \cup \Delta_{n-1}$.  The proof of the claim is then similar to the proof of Theorem~\ref{diagonal} above, and the theorem follows.
\end{proof}

\begin{proof}[Proof of Theorem~\ref{diagonal to Koszul}] Suppose the diagonal is compatibly split in $X \times X$.  Then $\Delta_1 \cup \cdots \cup \Delta_{n-1}$ is compatibly split in $X^n$ for every $n$, by Theorem~\ref{semidiagonals}.  Therefore, for any nef line bundles $L_1, \ldots, L_r$ on $X$ the section ring $R(L_1, \ldots, L_r)$ is normal and Koszul, by Proposition~\ref{Koszul}.
\end{proof}

\begin{remark}
Pairs of opposite lattice points $u$ and $-u$ in the polytopes associated to anticanonical divisors have also appeared in relation to the classification of smooth toric Fano varieties.  In particular, Ewald conjectured twenty years ago that if $X$ is a smooth toric Fano variety then $\F_X$ contains a basis for the character lattice $M$ \cite{Ewald88}.  Ewald's conjecture has been verified for smooth toric Fano varieties of dimension less than or equal to seven by {\O}bro \cite[Section~4.1]{OebroThesis}.  However, it remains unknown in higher dimensions whether there exists a single nonzero lattice point $u$ in $\F_X$ \cite[Section~4.6]{KreuzerNill07}.
\end{remark}

\bibliography{math}
\bibliographystyle{amsalpha}

\end{document}